\documentclass[a4paper, 10pt]{article}

\usepackage[a4paper,margin=2.5cm]{geometry}
\usepackage[english]{babel}
\usepackage{amsmath}
\usepackage{amssymb}
\usepackage{amsfonts}
\usepackage{amsthm}
\usepackage{cite}
\usepackage{hyperref}
\usepackage{color}
\usepackage{algorithm2e}
\usepackage[noblocks]{authblk}
\usepackage{mathtools}
\usepackage{hyphenat}

\newcommand{\R}{\mathbb{R}}
\newcommand{\N}{\mathbb{N}}
\newcommand{\calx}{\mathcal X}
\newcommand{\calh}{\mathcal H}
\newcommand{\fa}{\hbox{ for all }}
\newcommand{\inner}[2]{\ifthenelse{\equal{#2}{}}{\left\langle\cdot,\cdot\right\rangle_{#1}}{\left\langle#2\right\rangle_{#1}}}
\newcommand{\norm}[2]{\ifthenelse{\equal{#2}{}}{\left\|\cdot\right\|_{#1}}{\left\|#2\right\|_{#1}}}
\newcommand{\seminorm}[2]{\ifthenelse{\equal{#2}{}}{\left|\cdot\right|_{#1}}{\left|#2\right|_{#1}}}
\newcommand{\nsx}{\mathcal H_k(\calx)}
\DeclareMathOperator{\conv}{aco}
\DeclareMathOperator{\Lip}{Lip}
\DeclareMathOperator{\Span}{span}
\DeclareMathOperator*{\argmin}{arg\min}
\DeclareMathOperator*{\argmax}{arg\max}

\newtheorem{theorem}{Theorem}
\newtheorem{lemma}[theorem]{Lemma}
\newtheorem{definition}[theorem]{Definition} 
\newtheorem{remark}[theorem]{Remark}

\title{\vspace{-2.2cm}Refined rates of convergence for target-data dependent greedy generalized interpolation with Sobolev kernels}

\author{
Bernard Haasdonk \thanks{Institute of Applied Analysis and Numerical Simulation, University of Stuttgart, Germany},
Gabriele Santin \thanks{Department of Environmental Sciences, Informatics and Statistics, Ca' Foscari University of Venice, Italy},
Tizian Wenzel \thanks{Department of Mathematics, Ludwig-Maximilians-Universität München, Germany},
Daniel Winkle \thanks{Institute for Stochastics and Applications, University of Stuttgart, Germany}
}

\date{\today}

\begin{document}
\maketitle

\vspace{-0.7cm}\begin{abstract}
Greedy methods have recently been successfully applied to generalized kernel interpolation, or the recovery of a function from data stemming from the
evaluation of linear functionals, including the approximation of solutions of linear PDEs by symmetric collocation.

When applied to kernels generating Sobolev spaces as their native Hilbert spaces, some of these greedy methods can provide the same error guarantee of
generalized interpolation on quasi-uniform points. More importantly, certain target-data-adaptive methods even give a dimension- and 
smoothness\hyp{}independent 
improvement in the speed of convergence over quasi-uniform points, thus offering advantages for
high-dimensional problems.

These convergence rates however contain a spurious logarithmic term that limits this beneficial effect. The goal of this note is to remove this factor, and 
this is possible by using estimates on metric entropy numbers.
\end{abstract}

\section{Introduction}
We consider a strictly positive definite kernel $k:\calx\times\calx\to \R$ on a set $\emptyset\neq\calx\subset\R^d$, i.e., a symmetric function for which the 
matrix $(k(x_i, x_j)))_{i,j=1}^n\in\R^{n\times n}$ is positive definite for any set of pairwise distinct points $\{x_1, \dots, x_n\}\subset \calx$. Associated 
to this kernel there is a unique reproducing kernel Hilbert space $\nsx$~\cite{Saitoh2016}, also known as the kernel's native space, that is a Hilbert space 
of real-valued functions on $\calx$, where $k(\cdot, x)$ is the Riesz representer of the point evaluation functional at $x\in\calx$.
Of particular interest for this work are Sobolev kernels on sufficiently smooth bounded domains $\calx\subset\R^d$, namely those for which $\nsx$ is norm 
equivalent to $W_2^\tau(\calx)$, the $L_2$-Sobolev space of smoothness $\tau>d/2$. Concrete examples of such kernels are the families of Mat\'ern and Wendland 
kernels~\cite{Porcu2024,Wendland1995a}. 

Kernels can be used to interpolate functions from scattered point evaluations, yielding methods that have convenient computational properties and rigorous error estimates~\cite{Wendland2005}.
These theoretical and practical guarantees can be extended to generalized interpolation, where the data are given by the evaluation of finitely many linear functionals on the unknown solution.
A remarkable instance is approximating the solution of Partial Differential Equations (PDEs) by symmetric collocation~\cite{Fasshauer1997}, where the generalized interpolant is the function of minimal norm that satisfies the PDE, in the strong sense, in a finite number of points.

For both standard and generalized interpolation (or symmetric collocation), uniform point locations give worst-case optimal error rates in Sobolev spaces. 
Nevertheless, the error bounds and the computational efficiency of the methods are highly influenced by the number and location of the interpolation points.

To tackle this aspect, and building on earlier instances developed for pure interpolation~\cite{DeMarchi2005,SchWen2000,Mueller2009,Wenzel2022c}, recent works have extended greedy algorithms to generalized interpolation~\cite{Schaback2019a,Wenzel2025,Albrecht2024}.
These methods avoid an expensive global optimization of the interpolation points, and instead pick them incrementally by optimizing a selection rule.
For plain interpolation, the resulting approximants are known to have optimal error decay in Sobolev spaces~\cite{Santin2017x,Wenzel2021a,Wenzel2022c} and, when using adaptive selection rules, to even provide a dimension- and smoothness-independent improvement in the rate of convergence over interpolation on uniform points~\cite{Wenzel2022c,Santin2024a}.

These theoretical guarantees have recently been extended to symmetric collocation in~\cite{Wenzel2025}, with the adaptive schemes being particularly attractive 
for approximating the solution of high-dimensional PDEs.
However, due to a technical detail in the proof strategy, these convergence rates contain a logarithmic term that appears to be suboptimal, and our main result 
(Section~\ref{sec:convergence}, Theorem~\ref{th:theorem}) shows that this term can indeed be removed.
As our approach is general and not tied specifically to PDEs, we will formulate all results for generalized interpolation at large, as detailed in 
Section~\ref{sec:sym_coll}, and comment on the case of symmetric collocation of PDEs along the way. These advancements are possible by resorting to some 
properties of metric entropy numbers, for which the necessary background is introduced in Section~\ref{eq:entropy_numbers}. As the scope of this note is only 
theoretical and we do not introduce algorithmic novelties, we refer to the extensive experiments published in~\cite{Wenzel2025} for the numerical validation of 
the method.

\section{Generalized interpolation and the PDE-greedy algorithms}\label{sec:sym_coll}
We give a brief account of generalized kernel interpolation specifically focusing on Sobolev kernels, and refer to Chapter 16 in~\cite{Wendland2005} or 
to~\cite{Wenzel2025} for a 
detailed treatment.

Let $\calx\subset\R^d$ be a bounded set with a Lipschitz boundary, and for $I\in\N$ and $i\in\{1, \dots, I\}$, let $\calx_i \subset\calx$ be a bounded 
subset of $\R^{d}$ or a smooth compact manifold of dimension $d_i\leq d$~\cite{Lee2012}, with $d, d_i\in\N$.
Let $m_i\geq 0$ and $\tau > m_i + d_i/2$ for $i=1, \dots, I$, and consider linear and bounded (differential) operators $L_i: W_2^\tau(\calx) \to 
W_2^{\tau-m_i}(\calx_i)$, meaning that there is $c>0$\footnote{We assume that the constant $c>0$ is shared by all $L_i$'s for notational simplicity, instead of 
considering $c_1, \dots, c_I$.} and a (trace) inequality
\begin{equation}\label{eq:bounded_L}
\norm{W_2^{\tau-m_i}(\calx_i)}{L_i v} \leq c \norm{W_2^{\tau}(\calx)}{v}\;\;\fa v\in W_2^\tau(\calx).
\end{equation}
The assumption $\tau-m_i>d_i/2$ guarantees in particular that $L_i v\in W_2^{\tau-m_i}(\calx_i)\hookrightarrow C(\calx_i)$. 
Moreover, by assuming $m_i\geq 0$ we are including in the definition operators which raise the differentiability of $v$, such as some integral operators. 
Still, condition~\eqref{eq:bounded_L} needs to be satisfied.

Given data functions $f_i\in W_2^{\tau-m_i}(\calx_i)$, we consider the following problem:
\begin{equation}\label{eq:gen_pde}
\text{Find } u\in W_2^\tau(\calx) \text{ s.t. } L_i u(x) = f_i(x)\;\fa x \in \calx_i,\;\;i=1, \dots, I,
\end{equation}
which we assume to have a unique solution, even if uniqueness is not required for our analysis.
For example, the standard Poisson problem with Dirichlet boundary values on an open and bounded Lipschitz set $\Omega\subset\R^d$ can be represented by 
problem~\eqref{eq:gen_pde} with 
$I\coloneqq2$, $L_1 u\coloneqq-\Delta u$, $m_1\coloneqq 2$, $\calx_1\coloneqq\Omega$, $d_1\coloneqq d$, $L_2 u\coloneqq u$, $m_2\coloneqq 1/2$, 
$\calx_2\coloneqq\partial\Omega$, $d_2\coloneqq d-1$, $\calx\coloneqq\bar \Omega$, $\tau>2+d/2$. In this case the continuity assumption~\eqref{eq:bounded_L} is 
trivial for $i=1$, and it is a trace inequality for $i=2$. 
For this and more general PDEs, existence and uniqueness hold under quite standard conditions (see
e.g.~\cite{Evans2010}, or more specifically Assumption 1 in~\cite{Wenzel2025} for sufficient conditions in the case of elliptic, second order, Dirichlet boundary
value problems).

These operators induce a set of linear functionals $\Lambda\coloneqq \Lambda_{1}\cup \dots\cup \Lambda_{I}$ with
\begin{equation*}
\Lambda_{i}\coloneqq \{\delta_x\circ L_i: x\in \calx_i\},\;\;i=1,\dots,I,
\end{equation*}
and for a kernel $k$ with $\nsx\hookrightarrow W_2^\tau(\calx)$ these functionals are continuous on $\nsx$. Indeed, 
using $W_2^{\tau-m_i}(\calx_i)\hookrightarrow C(\calx_i)$, the inequality~\eqref{eq:bounded_L} , and $\nsx\hookrightarrow W_2^\tau(\calx)$, for $x\in \calx_i$ 
we have
\begin{align*}
\norm{\calh_k(\calx)'}{\delta_x\circ L_i}
&= \sup\limits_{0\neq v\in \nsx} \frac{|L_i v(x)|}{\norm{\nsx}{v}}
\leq \sup\limits_{0\neq v\in \nsx} \frac{\norm{C(\calx_i)}{L_i v}}{\norm{\nsx}{v}}
\leq C\cdot \sup\limits_{0\neq v\in \nsx} \frac{\norm{W_2^{\tau-m_i}(\calx_i)}{L_i v}}{\norm{\nsx}{v}}\\
&\leq C'\cdot \sup\limits_{0\neq v\in \nsx} \frac{\norm{W_2^{\tau}(\calx)}{v}}{\norm{\nsx}{v}}
\leq C'',
\end{align*}
where $C,C',C''>0$ are suitable constants.
These functionals have thus Riesz representers in $\nsx$, which are given by applying $\lambda\in\calh_k(\calx)'$ to the first kernel argument, 
i.e., $v_\lambda(x)\coloneqq\lambda(k(\cdot, x))\in\nsx$ for $x\in \calx$ (Theorem 16.7 in~\cite{Wendland2005}).
We can then define the bounded set $K\coloneqq \{v_\lambda: \lambda \in\Lambda\} = K_1\cup \dots\cup K_I$, where
\begin{equation}\label{eq:set_K}
K_i \coloneqq \{v_\lambda: \lambda \in\Lambda_{i}\} = \{L_i(k(\cdot, x)): x\in\calx_i\}\subset\nsx,\;\;i=1,\dots, I.
\end{equation}
To solve~\eqref{eq:gen_pde} by generalized interpolation we choose $n\coloneqq n_1+\dots+n_I$ points $X\coloneqq X_1\cup \dots\cup X_I\subset\calx$, where
$X_i\coloneqq\{x_{i,j}\}_{j=1}^{n_i}\subset\calx_i$ with $n_i\in\N$, and we select the corresponding functionals $\lambda_{i, j}\coloneqq \delta_{x_{i,j}} 
\circ L_i\in 
\Lambda_{i}$, $1\leq j\leq n_i$.
The approximate solution of~\eqref{eq:gen_pde} is then defined as
\begin{align}\label{eq:min_prob}
u_n&\coloneqq \argmin_{v\in \nsx}\norm{\nsx}{v}\\
&\text{ s.t. }\lambda_{i, j}(v) = f_i(x_j),\;\;1\leq j\leq n_i,\;\; i=1,\dots, I.\nonumber
\end{align}
This solution can be shown to take the generic form
\begin{equation}\label{eq:gen_interp}
u_n(x)\coloneqq \sum_{i=1}^I \sum_{j=1}^{n_i} \alpha_{i,j} v_{\lambda_{i,j}}(x),
\end{equation}
with coefficients determined by imposing the constraints in~\eqref{eq:min_prob} on $u_n$.
If the Riesz representers generating the space $V^n\coloneqq \Span\{v_{\lambda_{i,j}}: 1\leq j\leq n_i, i=1,\dots, I\}\subset K$ are additionally assumed to be 
linearly independent, then the representation~\eqref{eq:gen_interp} is unique, and the coefficients can be found by solving an invertible linear system. 
This linear independence is guaranteed for commonly used differential operators if the points in each $X_i$ are pairwise disjoint (see Theorem 
16.4 in~\cite{Wendland2005}), and we will see in Lemma~\ref{lemma:lin_ind} that this is always the case for certain greedy methods.
In any case, it can be proven that $u_n = \Pi_{V^n}(u)$, with $\Pi_{V^n}$ being the $\nsx$-orthogonal projector onto $V^n$.

Bounds on the error $u-u_n$ are usually derived by combining bounds on the residuals $L_i(u-u_n)$ with a stability result for the problem~\eqref{eq:gen_pde}. 
More specifically, we assume that there is $C_s>0$ such that for all $v, v'\in W_2^\tau(\calx)$ it holds
\begin{equation}\label{eq:continuity}
\norm{L_\infty(\calx)}{v-v'}\leq C_s\left( \norm{L_\infty(\calx_1)}{L_1(v-v')}+\dots + \norm{L_\infty(\calx_I)}{L_I(v-v')}\right).
\end{equation}
This assumption is satisfied by a maximum principle if~\eqref{eq:gen_pde} is an elliptic PDE (as used e.g. in~\cite{Wendland2005} and~\cite{Wenzel2025}), but 
more general
bounds of this type are commonly employed in linear and nonlinear collocation methods (see~\cite{Schaback2007}, equation (1.2), for a general framework, and
e.g.~\cite{Lee2025,Xu2025} for recent examples).
Observe moreover that~\eqref{eq:continuity} implies a.e.-uniqueness of the solutions of~\eqref{eq:gen_pde}, thus uniqueness since 
$W_2^\tau(\calx)\hookrightarrow C(\calx)$.

\begin{remark}
The framework of this section covers also the solution of parametric PDEs via symmetric collocation~\cite{Haasdonk2025a}, by letting $\hat x \coloneqq (x,\mu) 
\in 
\hat \calx \coloneqq \calx \times \calx_\mu$, where $\calx_\mu$ is a space of parameters and the differential operators only act on the position subvector $x$. 
For this, we need to assume 
slice-wise well-posedness results in the sense that~\eqref{eq:continuity} holds for any slice given by $\calx(\mu)\coloneqq \calx \times \{\mu\}$ (see the 
bound (10) in~\cite{Haasdonk2025a}).
\end{remark}

\subsection{Greedy generalized interpolation}
Greedy methods applied to generalized interpolation keep $u_n$ in the form~\eqref{eq:gen_interp} with coefficients satisfying the constraints in~\eqref{eq:min_prob}, but now restricting the
sum to an
incrementally selected set of functionals $\Lambda^n\coloneqq\{\lambda_j: 1\leq j\leq n\}\subset\Lambda$, or equivalently to collocation points 
$X^{n}\coloneqq\{x_{i_\ell j_\ell}: 1\leq \ell\leq n\}\subset\calx$, instead of a-priori prescribing the location and number of the points $X_1\subset\calx_1, 
\dots, 
X_I\subset\calx_I$.

Letting $\beta\geq 0$, $V^0\coloneqq \{0\}$ and $V^n \coloneqq \Span\{v_{\lambda_j}: \lambda_j\in \Lambda^n\}\subset K$, we are particularly interested in 
functionals selected by the \emph{PDE $\beta$-greedy} selection criteria, which are given in Definition 4.2
in~\cite{Wenzel2025} as
\begin{equation}\label{eq:selection}
\lambda_{n} \in \argmax_{\lambda \in \Lambda\setminus \Lambda^{n-1}} \eta_{n-1,\beta}(\lambda),\;\;
\eta_{n,\beta}(\lambda)\coloneqq
 \left|\lambda(u - \Pi_{V^n}(u))\right|^\beta\cdot P_{V^n}^{1-\beta}(\lambda),\;\;n\in\N,
\end{equation}
with the generalized power function (Section 16.1 in~\cite{Wendland2005}) defined by
\begin{equation}\label{eq:gen_pf}
P_{V^n}(\lambda)
\coloneqq \sup\limits_{\norm{\nsx}{v}\leq 1}|\lambda(v - \Pi_{V^n}(v))|
=\norm{\nsx}{v_\lambda - \Pi_{V^n}(v_\lambda)}.
\end{equation}
These criteria can be extended also to the limit $\beta\to\infty$ (obtaining a so called PDE-$f/P$-greedy algorithm), but for this case the analysis
of~\cite{Wenzel2025} is already exhaustive and thus we do not consider it here. 

Before proceeding, we show that any PDE $\beta$-greedy method always selects linearly independent functionals, as claimed above, provided it is stopped 
whenever $\eta_{n,\beta}(\lambda)=0$ for all $\lambda\in \Lambda$, in which case $\lambda_1, \dots, \lambda_n$ are a finite basis of $\Lambda$.
\begin{lemma}\label{lemma:lin_ind}
Let $n\in\N$ and $\beta\geq 0$.
If $\eta_{n,\beta}(\lambda)\neq0$ for some $\lambda\in \Lambda$, then $\lambda_{n+1}$ is linearly independent from $\lambda_1, \dots, \lambda_n$.
If instead $\eta_{n,\beta}(\lambda)=0$ for all $\lambda\in \Lambda$, then $\Lambda\subset\Span\{\lambda_1, \dots, \lambda_n\}$ (or equivalently $K\subset V^n$).
\end{lemma}
\begin{proof}
Equation~\eqref{eq:gen_pf} implies that $P_{V^n}(\lambda)=0$ if and only if $v_\lambda\in V^n$.
On the other hand, $\lambda(u - \Pi_{V^n}(u))=\inner{\nsx}{v_\lambda, u - \Pi_{V^n}(u)}=0$ if and only if $v_\lambda\in ((V^n)^\perp)^\perp = V^n$, since $V^n$ 
is finite dimensional thus closed.

Now if $\eta_{n,\beta}(\lambda)=0$ then $\lambda(u - \Pi_{V^n}(u))=0$ for $\beta>1$, or possibly even $P_{V^n}(\lambda)=0$ if $\beta\in[0,1]$, and thus in both 
cases $v_\lambda \in V^n$ or equivalently $\lambda\in\Span\{\lambda_1, \dots, \lambda_n\}$.
Viceversa, for all $\beta\geq 0$ if $\eta_{n,\beta}(\lambda)>0$ then $\lambda(u - \Pi_{V^n}(u))\neq0$, thus $v_\lambda \notin V^n$, implying the linear independence of the corresponding functional.
\end{proof}

We refer to~\cite{Wenzel2025} for details of the practical implementation of this method and for a discussion of its computational efficiency, 
and to~\cite{Haasdonk2025a} for a more stable implementation for the case $\beta>1$. In particular, we remark that the iteration 
is always stopped at a finite $n$ based on some termination criteria.

\section{Entropy numbers and generalized interpolation}\label{eq:entropy_numbers}

We review some notions related to entropy numbers, for which we refer to~\cite{Carl1981} and to Appendix A.5.6 in~\cite{Steinwart2008}, and recall their
use in the error theory of greedy algorithms following~\cite{Li2023,Santin2024a}\footnote{The notation in our paper partially updates that
of~\cite{Santin2024a} to better reflect the standard notation used e.g. in~\cite{Carl1981}.}.

For a Hilbert space\footnote{These definitions can be given for a metric space $V$. We stick to this simpler case as it suffices for our purposes.} $(V, 
\norm{V}{})$, a bounded subset $M\subseteq V$, and $n\in\N$, the $n$-th
dyadic metric entropy number of $M$ is defined by
\begin{equation}\label{eq:entropy}
e_n(M)
\coloneqq e_n(M, \|\cdot\|_V)
\coloneqq\inf\left\{\varepsilon > 0: \exists\ x_1, \dots, x_{q}\in V, q\leq 2^{n-1}: M\subseteq \bigcup_{i=1}^q B_V(x_i, \varepsilon)\right\},
\end{equation}
where $B_V(x_i, \varepsilon)$ is the closed $V$-ball centered at $x_i\in V$ with radius $\varepsilon>0$.
We remark that these balls can be equivalently chosen to be open, and that the points $x_1, \dots, x_q$ can be restricted to be in $M$ up to
multiplying $e_n(M)$ by $2$ (equation A.36 in~\cite{Steinwart2008}).
In particular, the entropy numbers of $M$ coincide with those of its closure $\overline M$ in $V$, hence in the following we need not to care if $M$ is closed or not.
The number $e_n(M)$ can be regarded as a measure of the precompactness of the set $M$, and in particular $\lim_{n\to\infty} e_n(M)=0$ if and only if 
$M$ is precompact in $V$ (equation (1.1.5) in~\cite{Carl1981}).
Furthermore, entropy numbers are  additive and increasing: Given bounded subsets $M, \hat M\subseteq V$, property (DE2) in~\cite[Section 1.3]{Carl1981}, gives 
that
\begin{equation}\label{eq:en_additive}
e_{n}(M +\hat M)
\leq e_\ell(M) + e_{n-\ell+1}(\hat M)\fa \ell\in\N, \ell \leq n,
\end{equation}
and if additionally $M\subseteq \hat M$ then clearly
\begin{equation}\label{eq:en_increasing}
e_n(M)\leq e_n(\hat M),\;\; n\in\N.
\end{equation}
The absolute convex hull of a subset $M\subseteq V$ (see e.g.~\cite{Carl2014}) is the convex and centrally
symmetric set
\begin{equation*}
\conv(M)
\coloneqq\left\{\sum_{v\in M} c_v v: c_v\in\R, \sum_{v\in M}|c_v|\leq 1\right\},
\end{equation*}
which is precompact if $M$ is precompact.
The rates of decay of $e_n(\conv(M))$ are well studied~\cite{Gao2004}, and for suitable sets $M$ they are relevant for the theory of greedy 
algorithms~\cite{Li2023}. 
In particular, they can be estimated when $M$ is \emph{smoothly parametrized} by the following class of Lipschitz maps (see~\cite{Lorentz1996a}, and Definition 
2 and 3 in~\cite{Siegel2022b}).
\begin{definition}\label{def:lipschitz_class}
For an open set $\Omega\subset\R^d$, a $V$-valued map $\mathcal F:\Omega\to V$, and any $v\in V$, consider the 
function
\begin{equation*}
\mathcal F_v(x)\coloneqq\inner{V}{v, \mathcal F(x)},\; x\in\Omega.
\end{equation*}
Given $s\coloneqq m+\alpha$ with $m\in\N_0$, $\alpha\in(0,1]$, we say that $\mathcal F\in\Lip_\infty(s, \Omega, V)$ if the partial derivative $D^\zeta \mathcal 
F_v$ is 
well defined for all multi-indices $\zeta\in\N_0^d$, $|\zeta|=m$, and if there is $C>0$ such that
\begin{equation*}
\seminorm{\Lip(s, L_\infty(\Omega))}{\mathcal F_v}
\coloneqq \max\limits_{|\zeta|=m}\sup_{\stackrel{y,z\in\Omega}{y\neq z}} \frac{\left|D^\zeta \mathcal F_v(y) - D^\zeta
\mathcal F_v(z)\right|}{|y-z|^\alpha}
\leq C \norm{V}{v}\;\;\fa v\in V.
\end{equation*}
For a smooth $d$-dimensional manifold $\calx$ and $\mathcal F:\calx\to V$, we say that $\mathcal F\in\Lip_\infty(s, \calx, V)$ if $\mathcal F\circ \varphi \in 
\Lip_\infty(s, \Omega, V)$ for each coordinate chart $(\Omega, \varphi)$\footnote{Assuming that $\calx$ is smooth implies that $\varphi\in 
C^\infty(\Omega;\R^d)$. We refer to Chapter 1 in~\cite{Lee2012} for details on smooth manifolds.}.
\end{definition}
With this definition we have the following result (Theorem 4 in~\cite{Siegel2022b}).
\begin{theorem}\label{th:siegel_bound}
Let $\calx$ be a $d$-dimensional compact smooth manifold, possibly with boundary.
Let $s>0$ and let $M$ be smoothly parametrized in the sense that $M\subseteq\mathcal F(\calx)$ with $\mathcal F\in\Lip_\infty(s, \calx, 
V)$.
Then
\begin{equation*}
e_n(\conv(M)) \leq c\cdot n^{-\frac sd -\frac12}, \;\;n\in\N,
\end{equation*}
where $c>0$ depends only on $s,d,\mathcal F$.
\end{theorem}
\begin{remark}\label{rem:bounded_set}
If $\calx\subset\R^d$ is a generic bounded set, then $\calx\subset \hat\calx\coloneqq B_{\R^d}(0, 
D)$ 
for some $D>0$, and $\hat\calx$ is a $d$-dimensional, smooth and compact manifold with boundary. It follows that $M\subseteq\mathcal F(\calx)\subseteq 
\mathcal F(\hat\calx)$ 
provided $\mathcal F$ can be extended to $\hat\calx$, and the theorem can still be applied if $\mathcal F\in\Lip_\infty(s, \hat\calx, V)$, and if $k$ is well
defined and positive definite also on $\hat\calx$, which is the case for commonly used kernels. In this case we have
$\calh_k(\calx)\hookrightarrow\calh_k(\hat\calx)$ with $\norm{\calh_k(\calx)}{f_{|\calx}}=\norm{\calh_k(\hat\calx)}{f}$ for all $f\in\calh_k(\hat\calx)$ 
(Theorem 10.46 in~\cite{Wendland2005}).
\end{remark}

\subsection{Convex hulls induced by linear and bounded differential operators}
We first prove how to bound the entropy number of the absolute convex hull of the union of $I$ sets.
\begin{lemma}\label{lemma:entropy_two_sets}
For $I\in\N$ and $i\in\{1,\dots, I\}$, let $K_i\subset V$ be bounded and assume that there are  $c_i,\alpha_i>0$ such that
\begin{equation}\label{eq:entropy_one_set}
e_{n}(\conv(K_i))\leq c_i n^{-{\alpha_i}}\;\;\fa n\in\N.
\end{equation}
Then for $K=K_1\cup \dots \cup K_I$ we have
\begin{equation}\label{eq:entropy_two_sets}
e_n(\conv(K))
\leq\left(\sum_{i=1}^I c_i I^{\alpha_i}\right)\cdot n^{-\min\{\alpha_1, \dots, \alpha_I\}}\;\fa n\geq I.
\end{equation}
\end{lemma}
\begin{proof}
First it holds that $\conv(K)\subset \conv(K_1) + \dots + \conv(K_I)$, since
any sum $f\coloneqq \sum_{v\in K} c_v v$  with $\sum_{v\in K}|c_v|\leq 1$ can be split (non uniquely) into $I$ sums, each over $v\in K_i$ and with similarly 
bounded coefficients. Thus~\eqref{eq:en_increasing} implies that
\begin{equation}\label{eq:ooo}
e_n(\conv(K)) \leq e_n(\conv(K_1) +\dots + \conv(K_I)).
\end{equation}
We assume now $n\geq I$ and use repeatedly the additivity~\eqref{eq:en_additive}, which gives for $\ell_i\in\N$ that
\begin{align}\label{eq:tmp}
e_{n}(\conv(K_1) +\dots + \conv(K_2))
&\leq e_{\ell_1}(\conv(K_1)) + e_{n-\ell_1+1}(\conv(K_2) +\dots + \conv(K_I))\nonumber\\
&\leq e_{\ell_1}(\conv(K_1)) + e_{\ell_2}(\conv(K_2)) + e_{n-\ell_1-\ell_2+2}(\conv(K_3) +\dots + \conv(K_I))\nonumber\\
&\leq \sum_{i=1}^{I-1}e_{\ell_i}(\conv(K_i)) +e_{n-\ell_1-\dots-\ell_{I-1}+(I-1)}(\conv(K_I)),
\end{align}
where by ~\eqref{eq:en_additive} we need to require $\ell_1\leq n$, $\ell_2\leq n-\ell_1+1$, $\dots$, $\ell_{I-1}\leq n - \ell_1-\dots -\ell_{I-2}+{I-2}$, 
i.e., $\ell_1+\dots + \ell_i \leq n + (i-1)$, $i=1, \dots, I-1$.

Choosing $\ell_i\coloneqq\lceil n/I\rceil$ for all $i=1, \dots, I-1$ satisfies these constraints, since
\begin{equation*}
\ell_1+\dots + \ell_i
= i \left\lceil \frac{n}{I}\right\rceil
\leq i \left(\frac nI + 1\right)
\leq n + (i-1) \Leftrightarrow i\leq \frac{n-1}{n}\cdot I,
\end{equation*}
and the last inequality is true for all $i=1, \dots, I-1$ if $n\geq I$. Moreover, this choice gives $\ell_i\geq n/I$ and
\begin{equation*}
n-\ell_1-\dots-\ell_{I-1}+(I-1)
= n - (I-1)\left( \left\lceil \frac n I\right\rceil - 1\right)
\geq n - (I-1) \frac n I
= \frac n I.
\end{equation*}
Since $e_n$ is decreasing in $n$, we can then insert these $\ell_i$'s into~\eqref{eq:tmp} and use~\eqref{eq:entropy_one_set} and~\eqref{eq:ooo} to get
\begin{equation*}
e_n(\conv(K))
\leq \sum_{i=1}^I c_i (n/I)^{-\alpha_i}
\leq \left(\sum_{i=1}^I c_i I^{\alpha_i}\right)\cdot n^{-\min\{\alpha_1, \dots, \alpha_I\}},
\end{equation*}
as claimed.
\end{proof}

Observe that the result is optimized for the asymptotic behavior $n\to\infty$, while a finer analysis could reveal different regimes for different values of 
$n$, depending on the values of the coefficients $c_i$ and $\alpha_i$, as it is usual in the study of entropy numbers. Moreover, replacing $\ell_i=\lceil 
n/I\rceil$ with an uneven distribution of the $\ell_i$'s may lead to a better constant in~\eqref{eq:entropy_two_sets}.

We now specialize Lemma~\ref{lemma:entropy_two_sets} to the case of interest. To prove the result, we first recall that if $\calx_i\subset\R^d$ is a smooth 
compact manifold of dimension $d_i\in\N$, and  $s_i\coloneqq \tau - d_i/2$, then there is a constant $C_e>0$ such that
\begin{equation}\label{eq:holder_sobolev_embedding}
\seminorm{\Lip(s_i, L_\infty(\calx_i))}{v}
\leq\norm{C^{s_i}(\calx_i)}{v}
\leq C_e \norm{W_2^\tau(\calx_i)}{v},
\end{equation}
where the first inequality comes from the definition of the norm of the H\"older space $C^{s_i}(\calx_i)$, and the second is a manifold version of the usual 
Sobolev 
embedding theorem (see e.g.~Proposition 3.3\footnote{More precisely, for $v\in W_2^\tau(\calx_i)$ equation (3.5) implies that $v\in C^{s_0}(\calx_i)$ with 
$s_0\coloneqq \lfloor \tau-d_i/2 \rfloor$. Then $D^\zeta v\in  W_2^{\tau-s_0}(\calx_i)$ for all $\zeta\in \N_0^{d_i}$, $|\zeta|\leq s_0$, and equation 
(3.6) can be applied to conclude.} in Chapter 4 of~\cite{Taylor1996}).
We have then the following bound on the entropy numbers.
\begin{lemma}\label{lemma:entropy_LB}
With the notation of Section~\ref{sec:sym_coll}, for $i=1, \dots, I$ let in particular $L_i$ satisfy~\eqref{eq:bounded_L} and let $K_i$ be defined as 
in~\eqref{eq:set_K}.
Then there is $c_i>0$ independent of $n$  such that
\begin{equation*}
e_n(\conv(K_i))\leq c_i \cdot n^{-\frac{\tau-m_i}{d_i}}\;\fa n\in\N,\;\;i=1,\dots, I.
\end{equation*}
Moreover for $K= K_1\cup\dots\cup K_I$ it holds that
\begin{equation}\label{eq:entropy_K}
e_n(\conv(K))\leq \bar C\cdot  n^{-(\tau-\bar m)/\bar d}\;\;\fa n\geq I,
\end{equation}
where $(\bar m, \bar d)\in \{(m_i, d_i), 1\leq i\leq I\}$ with $(\tau-\bar m)/{\bar d}= \min_{1\leq i\leq 
I}\left((\tau-m_i)/{d_i}\right)$, and $\bar C\coloneqq \sum_{i=1}^I c_i 
I^{\frac{\tau-m_i}{d_i}}$.\end{lemma}
\begin{proof}
We start with a single functional $L_i$ and follow the approach of Section 3.4 in \cite{Santin2024a}. We assume that $\calx_i$ is a smooth manifold, possibly 
after the extension described in Remark~\ref{rem:bounded_set}, and use~\eqref{eq:set_K} to write $K_i \subseteq \mathcal F_i(\calx_i)$ with $\mathcal 
F_i:\calx_i\to\nsx$, $\mathcal F_i(x)\coloneqq L_i(k(\cdot, x))$.

Let now $(\Omega, \varphi)$ be a coordinate chart and $\mathcal F_i'\coloneqq \mathcal F_i\circ \varphi$.
For $\omega\in \Omega$, $x\coloneqq \varphi(\omega)\in \calx_i$, and $v\in\nsx$ we get
\begin{equation*}
\mathcal F_v'(\omega)
\coloneqq \inner{\nsx}{v, \mathcal F'_i(\omega)}
=\inner{\nsx}{v, L_i(k(\cdot, \varphi(\omega)))}
= L_i v(\varphi(\omega)),
\end{equation*}
and using the chain rule we get for $s_i\coloneqq \tau - m_i - d_i/2$ that
\begin{equation*}
\seminorm{\Lip(s_i, L_\infty(\Omega))}{\mathcal F_v'}
\leq c_\varphi\seminorm{\Lip(s_i, L_\infty(\calx_i))}{L_i v},
\end{equation*}
where $c_\varphi>0$ depends on the $C^{s_i}(\Omega)$-norms of the local coordinate functions $\varphi_1, \dots, \varphi_{d_i}$. 
Using~\eqref{eq:holder_sobolev_embedding} and~\eqref{eq:bounded_L} we further get
\begin{equation*}
\seminorm{\Lip(s_i, L_\infty(\calx_i))}{L_i v}
\leq\norm{C^{s_i}(\calx_i)}{L_i v}
\leq C_e \norm{W_2^{\tau-m_i}(\calx_i)}{L_i v}
\leq C_e c \norm{W_2^{\tau}(\calx)}{v}
\leq C'\norm{\nsx}{v},
\end{equation*}
using the embedding $\nsx\hookrightarrow W_2^\tau(\calx)$ in the last step.
It follows that
\begin{equation*}
\seminorm{\Lip(s_i, L_\infty(\Omega))}{\mathcal F_v'}
\leq C' \norm{\nsx}{v}\;\;\fa v\in\nsx,
\end{equation*}
which means that $\mathcal F_i'=\mathcal F_i\circ \varphi\in \Lip_{\infty}(s_i,\Omega, \nsx)$ for each $(\Omega, \varphi)$, implying $\mathcal F_i\in 
\Lip_{\infty}(s_i,\calx_i, \nsx)$ by Definition~\ref{def:lipschitz_class}.
Recalling that each $\calx_i$ is a smooth, $d_i$-dimensional manifold, Theorem~\ref{th:siegel_bound} applied to $K_i \subseteq \mathcal F_i(\calx_i)$ 
gives that there is $c_i>0$ depending only on $s_i, d_i, \mathcal F_i$ such that
\begin{equation*}
e_n(\conv(K_i))
\leq c_i\cdot n^{-\frac{s_i}{d_i} -\frac12}
= c_i\cdot n^{-\frac{\tau-m_i-d_i/2}{d_i} -\frac12}
= c_i\cdot n^{-\frac{\tau-m_i}{d_i}}.
\end{equation*}
The bound~\eqref{eq:entropy_K} then follows by applying Lemma~\ref{lemma:entropy_two_sets}.
\end{proof}

\section{Convergence of PDE-greedy algorithms}\label{sec:convergence}
We now link the product of incremental power function evaluations to the decay of the entropy numbers.
The argument is based on Lemma 2.1 in~\cite{Li2023}, which estimates the volume of a simplex in $K$ by a certain function of $e_n(\conv(K))$
\begin{lemma}\label{lemma:pf_prod}
Under the assumption of Lemma~\ref{lemma:entropy_LB}, let $\Lambda^j\coloneqq\{\lambda_1, \dots, \lambda_j\}\subset\Lambda$ be any nested set of functionals.
Then
\begin{equation}\label{eq:pf_prod_bound}
p_n\coloneqq\left(\prod_{j=n+1}^{2n}P_{\Lambda_j}(\lambda_{j+1})\right)^{1/n}
\leq \sqrt{5} n^{\frac12} e_n(\conv(K))
\leq \sqrt{5} \cdot \bar C\cdot  n^{-\frac{\tau-\bar m}{\bar d} + \frac12}\;\;\fa n\geq I,
\end{equation}
where $\bar C, \bar m, \bar d>0$ are as in Lemma~\ref{lemma:entropy_LB}.
\end{lemma}
\begin{proof}
Using~\eqref{eq:gen_pf}, for $i=1, \dots, 2n$ we have that $P_{\Lambda_i}(\lambda_{i+1})$ is the distance between $v_{\lambda_{i+1}}$ and the subspace 
$V^i=\Span\{v_{\lambda_j}:1\leq j\leq i\}$, where $v_{\lambda_{j}}\in K$ for $j=1, \dots, i+1$.
We can thus apply Lemma 2.1 in~\cite{Santin2024a} with $m=n$, to get that
\begin{equation*}
p_n \leq \sqrt{5} n^{1/2} e_n(\conv(K)),
\end{equation*}
and inserting the bound of Lemma~\ref{lemma:entropy_LB} concludes the proof.
\end{proof}

We finally obtain our main result. The argument follows the exact same proof of Theorem 5.1 in~\cite{Wenzel2025}, where Corollary 4.2 
in~\cite{Wenzel2025} is replaced by the sharper bound of Lemma~\ref{lemma:pf_prod}. For a related result in the case of Hermite-Birkhoff interpolation, we also refer to Section 2.1 in~\cite{Herkert2026}.
\begin{theorem}\label{th:theorem}
Under the assumption of Lemma~\ref{lemma:entropy_LB}, for $j=1, \dots, 2n$ let $u_j$ be the generalized interpolant based on the functionals selected by a 
$\beta$-greedy algorithm with $\beta\geq 0$.
Then
\begin{equation}\label{eq:error_bound}
\min_{n+1\leq j \leq 2n} \norm{L_\infty(\calx)}{u - u_j}
\leq C' \left(n^{-\frac{\tau - \bar m}{\bar d}+\frac{1-\beta}{2}}\right)^{\frac{1}{\max(\beta,1)}}\cdot \norm{\nsx}{u - u_{n+1}} \;\;\fa 
n\geq I,
\end{equation}
where $C'\coloneqq C_s\cdot I\cdot \left(\sqrt{5} \cdot \bar C\right)^{1/(\max(1,\beta))} $, and $\bar C, \bar m, \bar d>0$ are as in 
Lemma~\ref{lemma:entropy_LB}.
\end{theorem}
\begin{proof}
Let $E_j\coloneqq u -u_j$.
The stability bound~\eqref{eq:continuity} and the definition of $\Lambda$ give
\begin{equation*}
\min_{n+1\leq j \leq 2n} \norm{L_\infty(\calx)}{E_j}
\leq C_s\cdot\min_{n+1\leq j \leq 2n} \sum_{i=1}^I \norm{L_\infty(\Omega_i)}{L_i(E_j)}
\leq C_s \cdot\min_{n+1\leq j \leq 2n} \sum_{i=1}^I \sup_{\lambda \in \Lambda} |\lambda(E_j)|.
\end{equation*}
Then we take from the proof of Theorem 5.1 in~\cite{Wenzel2025} the estimate
\begin{equation}\label{eq:tmp_proof_one}
\min_{n+1\leq j \leq 2n} \sup_{\lambda \in \Lambda} |\lambda(E_j)|
\leq \left(\prod_{j=n+1}^{2n} \sup_{\lambda \in \Lambda} |\lambda(E_j)| \right)^{1/n}
\leq 
n^{-\min(\beta,1)/2} \cdot \Vert E_{n+1} \Vert_{\nsx} \cdot p_n^{1/\max(1,\beta)}.
\end{equation}
The result follows by combining the two chains of inequalities, and inserting the bound~\eqref{eq:pf_prod_bound} for $p_n$. 
\end{proof}
We remark that, up to different constants, the result coincides with that of Theorem 5.1 in~\cite{Wenzel2025}, where we removed the factor 
$
\log(n)^{\frac{2(\tau - \bar m) - \bar d}{2\bar d \max(1,\beta)}}$ from the right-hand side.
Moreover, we can handle piecewise-smooth manifolds $\calx_i$ by simply defining a different operator $L_i$ on each piece. This 
enters the estimate~\eqref{eq:error_bound} by increasing $C'$ through an increased value of $I$, while the estimates in~\cite{Wenzel2025} do not have a 
similar penalty.

\section{Discussion and conclusion}
In this note we revised the proof of convergence for the PDE-$\beta$-greedy algorithms for Sobolev kernels and $\beta\geq 0$ presented in~\cite{Wenzel2025}.
For the particular case of $\beta\in[0,1]$, our main result (Theorem~\ref{th:theorem}) gives the bound
\begin{equation*}
\min_{n+1\leq j \leq 2n} \norm{L_\infty(\calx)}{u - u_j}
\leq C' n^{-\frac{\tau - \bar m}{\bar d}+\frac{1-\beta}{2}}\cdot \norm{\nsx}{u - u_{n+1}} \;\;\fa n\geq I,
\end{equation*}
which coincides with the rate of convergence of generalized interpolation with uniform points for $\beta=0$ ($P$-greedy, no target-adaptivity in the selection
rule), while it provides a dimension- and smoothness-independent improvement $n^{-\beta/2}$ for $\beta\in(0,1]$, which is maximized when adaptivity is 
fully exploited ($\beta=1$,  $f$-greedy). 

Our result is formulated for a generic problem~\eqref{eq:gen_pde} based on $I\in\N$ linear operators. 
As our bound follows from rather general continuity~\eqref{eq:bounded_L} and stability assumptions~\eqref{eq:continuity}, we envision its application to more 
general settings even beyond PDE-based problems.

\subsection*{Acknowledgments}
B.H. acknowledges funding by the Deutsche Forschungsgemeinschaft (DFG, German Research Foundation) under number 540080351 and Germany’s Excellence Strategy - 
EXC 2075 - 390740016.
G.S.~is a member of INdAM-GNCS, and his work was partially supported by the project ``Perturbation problems and asymptotics for elliptic differential 
equations: variational and potential theoretic methods'' funded by the program ``NextGenerationEU'' and by MUR-PRIN, grant 2022SENJZ3. 
D. W. acknowledges funding by the International Max Planck Research School for Intelligent Systems (IMPRS-IS).

\bibliography{bibexport}
\bibliographystyle{abbrv}

\end{document}